\documentclass{amsart}
\usepackage{graphicx}

\usepackage[all]{xy}

\newtheorem{theorem}{Theorem}
\newtheorem{lemma}[theorem]{Lemma}
\newtheorem{corollary}[theorem]{Corollary}
\newtheorem{proposition}[theorem]{Proposition}
\theoremstyle{definition}

\newtheorem{definition}[theorem]{Definition}
\newtheorem{example}[theorem]{Example}
\newtheorem{question}[theorem]{Question}
\newtheorem{remark}[theorem]{Remark}

\numberwithin{theorem}{section}
\numberwithin{equation}{section}

\newcommand{\EQ}[1]{(\ref{#1})}
\newcommand{\SEC}[1]{Section \ref{#1}}
\newcommand{\THM}[1]{Theorem \ref{#1}}
\newcommand{\LEM}[1]{Lemma \ref{#1}}
\newcommand{\PROP}[1]{Proposition \ref{#1}}

\newcommand{\R}{{\mathbb{R}}}
\newcommand{\C}{{\mathbb{C}}}
\newcommand{\ep}{\varepsilon}
\newcommand{\Lip}{\operatorname{Lip}}
\newcommand{\trace}{\operatorname{trace}}

\newcommand{\mrbox}[1]{ \quad \mbox{#1} \ }
\newcommand{\mmbox}[1]{ \quad \mbox{#1} \quad }
\newcommand{\Lg}{\text{S}}
\newcommand{\Le}{\hat{\text{S}}}
\newcommand{\Lu}{\text{L}}

\renewcommand{\div}{\operatorname{div}}

\begin{document}

\title{Vector-valued optimal Lipschitz extensions}

\author{Scott Sheffield}
\address{Massachusetts Institute of Technology}
\email{sheffield@math.mit.edu}

\author{Charles K. Smart}
\address{University of California Berkeley}
\email{smart@math.berkeley.edu}

\date{\today}

\begin {abstract}
Consider a bounded open set $U \subset \R^n$ and a Lipschitz function $g: \partial U \to \R^m$.  Does this function always have a canonical optimal Lipschitz extension to all of $U$?  We propose a notion of optimal Lipschitz extension and address existence and uniqueness in some special cases. In the case $n=m=2$, we show that smooth solutions have two phases: in one they are conformal and in the other they are variants of infinity harmonic functions called \emph{infinity harmonic fans}.  We also prove existence and uniqueness for the extension problem on finite graphs.
\end {abstract}

\maketitle


\section{Introduction}

\subsection{Overview}

Suppose $U \subseteq \R^n$ is bounded and open and $g : \partial U \to \R^m$ is Lipschitz, i.e.,
$\Lip(g,\partial U)<\infty,$ where
$$\Lip(f,X):= \sup_{x,y \in X} \frac{|f(x)-f(y)|}{|x-y|}.$$
A classical theorem of Kirszbraun implies that $g$ has an extension $u : \bar U \to \R^m$ such that $\Lip(u,\bar U) = \Lip(g,\partial U)$.  In general, there may be infinitely many such extensions \cite{kirszbraun}.

When $m = 1$, a theorem of Jensen \cite{MR1218686} implies that there is a unique extension $u$ of $g$ that is \emph{absolutely minimizing Lipschitz (AML)}, i.e., a unique extension $u$ satisfying
\begin{equation}
\label{defAML}
\Lip(u,V) = \Lip(u, \partial V) \mrbox{for every open} V \subseteq U.
\end{equation}
The situation is more complicated when $m > 1$, as the criterion \EQ{defAML} fails to characterize a unique extension. Indeed, there are functions satisfying \EQ{defAML} whose Lipschitz derivatives can be decreased everywhere (see Section \ref{secTight}). The purpose of this article is to identify and study an appropriate notion of optimal Lipschitz extension when $m > 1$.

We remark that the general problem of finding Lipschitz extensions of a Lipschitz function from a subset $Y$ of a metric space $X$ to another metric space $Z$ has been thoroughly studied, with diverse applications in mathematics, computer science, and engineering (see the books~\cite{MR0461107,BL00}, as well as the introductions of~\cite{MR2129708,MR2239346,naorsheffield}, and the references therein).  Much of this work has focused on finding extensions whose Lipschitz norm is either equal to or within some constant factor of the minimal one.  The idea of imposing conditions that would identify a single {\em uniquely} optimal extension has not received as much attention except in the following cases:
 \begin{enumerate}
 \item $Z = \R$.
 \item $Z$ is a metric tree.
 \item $X = \R$.
 \end{enumerate}
The first case is the subject of an extensive literature (see Section \ref{s.scalarhistory}).  In particular, when $Z=\R$ and $X$ is any length space, it is known that every bounded Lipschitz function on a closed subset of $X$ admits a unique AML extension to $X$ \cite{MR2449057}. In the second case, a more recent paper shows existence and uniqueness of AML extensions whenever $X$ is a locally compact length space \cite{naorsheffield}. In the third case, AML extensions are geodesic paths, which exist by definition if $Z$ is a geodesic metric space.  It is natural to seek optimal Lipschitz extensions for functions between more general pairs of metric spaces.  We limit our attention in this paper to the case that $Z = \R^m$ and either $X \subseteq \R^n$ is the closure of a bounded domain (and $Y$ is the boundary) or $X$ is a finite graph.  (The case that $X \subseteq \R^n$ and $Z = \R^m$ is particularly natural in light of Kirszbraun's theorem, as stated above.)

This paper introduces a notion of optimality called {\em tightness} that is stronger than the AML property \eqref{defAML} and yields existence and uniqueness results in the discrete setting, when $X$ is a finite graph. We extend the definition of tight extension to the continuum setting and then partially characterize the smooth tight functions using PDE.   In the case $n=m=2$, we prove that the smooth tight functions come in two ``phases''. In one phase they are conformal and in the other they are variants of infinity harmonic functions called \emph{infinity harmonic fans}.  We also prove some existence and uniqueness results within a class of radially symmetric tight functions.  Although AML extensions are known to be $C^1$ when $n=2$ and $m=1$ \cite{MR2185662}, our examples show that this is not necessarily true of tight extensions (recall that tight implies AML) when $n=m=2$.



\subsection{Tight functions on finite graphs}

We begin by studying the discrete version of our problem. Let $G = (X, E, Y)$ denote a connected graph with vertex set $X$, edge set $E$, and a distinguished non-empty set of vertices $Y \subseteq X$. We wish to understand when a function $u : X \to \R^m$ is the optimal Lipschitz extension of its restriction $u |_Y$ to $Y$.

The \emph{local Lipschitz constant} of a function $u : X \to \R^m$ at a vertex $x \in X \setminus Y$ is given by
\begin{equation*}
\Lg u(x) := \sup_{y \sim_E x} |u(y) - u(x)|.
\end{equation*}
A function $u : X \to \R^m$ is {\em discrete infinity harmonic} at $x \in X \setminus Y$ if there is no way to decrease $\Lg u(x)$ by changing the value of $u$ at $x$. It was shown in \cite{MR2449057} that if $G$ is finite then every $g : Y \to \R$ has a unique extension $u : X \to \R$ that is discrete infinity harmonic on $X \setminus Y$.

We show that uniqueness of discrete infinity harmonic extensions fails in the vector-valued case. To obtain uniqueness, we adopt a stronger optimality criterion.

\begin{definition}
If the functions  $u, v : X \to \R^m$ agree on $Y$ and satisfy
\begin{equation}
\label{defTightG}
\sup \{ \Lg u : \Lg u > \Lg v \} > \sup \{ \Lg v : \Lg u < \Lg v \},
\end{equation}
then we say that $v$ is \emph{tighter} than $u$ on $G$. We say that $u$ is \emph{tight} on $G$ if there is no tighter $v$. Informally, $v$ is tighter if it improves the local Lipschitz constant where it is large without making it too much worse elsewhere.
\end{definition}

We prove existence and uniqueness of tight extensions for finite graphs. The uniqueness of tight extensions fails for some infinite graphs; see \PROP{nonunique}.


\begin{theorem}
\label{EUG}
Suppose $G = (X, E, Y)$ denotes a finite connected graph with vertex set $X$, edge set $E$, and a distinguished non-empty set of vertices $Y \subseteq X$. Every function $g : Y \to \R^m$ has a unique extension $u : X \to \R^m$ that is tight on $G$.  Moreover, this $u$ is tighter than every other extension $v : X \to \R^m$ of $g$.
\end{theorem}

When the graph is finite, we prove that the unique tight extension is the limit of the discrete $p$-harmonic extensions.

\begin{theorem}
\label{PlimG}
In addition to the hypotheses of \THM{EUG}, suppose that for each $p > 0$, the function $u_p : X \to \R^m$ minimizes
\begin{equation}
\label{penergyG}
I_p[u] := \sum_{x \in X \setminus Y} (\Lg u(x))^p,
\end{equation}
subject to $u_p |_Y = g$. As $p \to \infty$, the $u_p$ converge to the unique tight extension of $g$.
\end{theorem}

Finally, we provide an equivalent definition of tight by replacing the supremum over vertices in \eqref{defTightG} with supremum over edges; see \PROP{edgeTight} in \SEC{secGraph}.


\subsection{Tight functions on $\R^n$}

Moving to the continuum case, we suppose $U \subseteq \R^n$ is bounded and open. Recall from \cite{MR1861094} that the AML extension of a scalar function $g \in C(\partial U)$ can also be characterized as the unique viscosity solution of the boundary value problem
\begin{equation}
\label{inflap}
\left\{ \begin{array}{ll}
- \Delta_\infty u = 0 & \mrbox{in} U, \\
u = g & \mrbox{on} \partial U,
\end{array} \right.
\end{equation}
where $\Delta_\infty u := |Du|^{-2} u_{x_i} u_{x_j} u_{x_i x_j}$ is the {\em infinity Laplacian}. Viscosity solutions of \eqref{inflap} are called {\em infinity harmonic}. This alternate name owes its origin to the fact that \EQ{inflap} is the limiting Euler-Lagrange equation as $p \to \infty$ for minimizers of the functional
\begin{equation}
\label{penergy}
I_p[u] := \int_U |Du|^p \,dx.
\end{equation}
Such minimizers are called \emph{$p$-harmonic} and are harmonic when $p=2$.  Moreover, for each $p > 1$ the extension of $g$ minimizing \EQ{penergy} is unique, and as $p \to \infty$ these minimizers converge pointwise to the infinity harmonic extension.

In the vector-valued case, we expect the limit as $p \to \infty$ of minimizers of \eqref{penergy} to be an optimal Lipschitz extension. However, in order to obtain something like an AML extension in the limit as $p \to \infty$, we must be careful to use the matrix norm
\begin{equation}
\label{defNorm}
|A| := \max_{|a| = 1} |A a|,
\end{equation}
so that
\begin{equation*}
\Lip(u, V) = \sup_V |Du|,
\end{equation*}
whenever $u \in C^1(U, \R^m)$ and $V \subseteq U$ is open. Indeed, if one uses the more usual Frobenius norm
\begin{equation*}
|A|_F = \trace(A^t A)^{1/2},
\end{equation*}
in \eqref{penergy}, then one obtains a notion of $p$-harmonic and infinity harmonic \cite{MR2539666, ou-2008} that is not related to AML extensions when $m > 1$. Because the norm \eqref{defNorm} is neither smooth nor strictly convex, trying to compute an optimal Lipschitz extension as the limit of $p$-harmonic extensions seems difficult in the vector-valued case.

Instead, we obtain a notion of optimal Lipschitz extension in the continuum case by taking the strongest natural analogue of \EQ{defTightG}. Consider a bounded open set $U \subseteq \R^n$. If $u \in C(U, \R^m)$ is Lipschitz, then the \emph{local Lipschitz constant} of $u$ at a point $x \in U$ is
\begin{equation*}
\Lu u(x) := \inf_{r > 0} \Lip(u, U \cap B_r(x)).
\end{equation*}

\begin{definition}
If the Lipschitz functions $u, v \in C(\bar U, \R^m)$ agree on $\partial U$ and satisfy
\begin{equation*}
\sup \{ \Lu u : \Lu v < \Lu u \} > \sup \{ \Lu v : \Lu v > \Lu u \},
\end{equation*}
then we say that $v$ is \emph{tighter} than $u$ on $U$. We say that $u$ is \emph{tight} if there is no tighter $v$.
\end{definition}

Our first result for tight functions on $\R^n$ is the following theorem. A \emph{principal direction field} for a function $u \in C^1(U, \R^m)$ is a unit vector field $a \in C(U, \R^n)$ such that $a(x)$ spans the principal eigenspace of $Du(x)^t Du(x)$ for all $x \in U$ (in particular, the principal eigenvalue of $Du(x)^t Du(x)$ is simple). Note that the linearization of $u$ at $x$ sends a small sphere centered at $x$ to a small ellipsoid centered at $u(x)$. If $a$ is a principal direction field, then $a(x)$ is parallel to the line whose image in $u$ is the (unique) major axis of that ellipsoid. Informally, $\pm a(x)$ are the directions in which $u$ is most rapidly changing.

\begin{theorem}
\label{Pr}
Suppose $U \subseteq \R^n$ is bounded open and $u \in C^3(U, \R^m)$ has principal direction field $a \in C^2(U, \R^n)$. Then $u$ is tight if and only if
\begin{equation}
\label{pdePr}
-(u_a)_a = 0 \mrbox{in} U,
\end{equation}
where we define $v_a := v_{x_i} a^i$ for any $v \in C^1(U, \R^k)$.
\end{theorem}

Note that in the case $m = 1$, we have $a = \pm |Du|^{-1} Du$ and thus $a \cdot a_a = 0$ and
\begin{equation*}
- (u_a)_a = - (u_{x_i} a^i)_{x_j} a^j = - u_{x_i x_j} a^i a^j  - u_{x_i} a^i_{x_j} a^j = - \Delta_\infty u.
\end{equation*}
Thus in the scalar case the system \eqref{pdePr} reduces to the PDE \eqref{inflap}. We show below that \eqref{pdePr} is the limiting Euler-Lagrange equation for minimizers of \eqref{penergy} as $p \to \infty$ in the principal case.

An immediate and important corollary of \THM{Pr} is that $C^3$ infinity harmonic functions from $\R^n$ to $\R$ are tight (recall that $C^2$ infinity harmonic functions have non-vanishing gradients \cite{MR2255237}). Equally important is the following geometric interpretation. We call the maximal curves in $U$ that are parallel to the principal direction field the \emph{principal flow curves}.  The system \EQ{pdePr} says that the image of a principal flow curve is a straight line. Moreover, it says that the map from the curve to the line has constant speed. When $n=2$, we can identify points on the image lines by the principal of orthogonal motion (see \SEC{secFans}) and obtain from $u$ an infinity harmonic function with a one-dimensional image.  In this case, $u$ may be interpreted as a ``fanned out'' version of a one-dimensional infinity harmonic function.

When the domain $U$ is a subset of $\R^2$ and the function $u$ is differentiable, the alternative to having a principal direction at a point is to be conformal (or anti-conformal) at that point. In this case, tightness is equivalent to a differential inequality.

\begin{theorem}
\label{Co}
Suppose $U \subset \C$ is bounded and open and $u: U \to \C$ is analytic in a neighborhood of $\overline U$. Then $u$ is tight if and only if either
\begin{enumerate}
\item $u'' = 0$ throughout $U$ (i.e., $u$ is affine on $U$).
\item The meromorphic function $\frac{u'u'''}{(u'')^2}$ satisfies \begin{equation}\label{e.uderivatives} \Re \frac{u'u'''}{(u'')^2} \leq 2\end{equation} wherever it is defined in $U$.  (In particular, this implies that any singularities of $\frac{u'u'''}{(u'')^2}$ are removable.)
\end{enumerate}
\end{theorem}

When $u$ is one-to-one and $u''$ does not vanish, the criterion \EQ{e.uderivatives} can be more suggestively written as
\begin{equation}
\label{e.delta1} - \Delta_1 \log |v'| \geq 0,
\end{equation}
where $v = u^{-1}$ and $\Delta_1 := \Delta - \Delta_\infty$. Informally, the inequality \EQ{e.delta1} says that the boundary of the set $\{ |v'| \geq s \}$ is locally convex for all $s > 0$. See Lemma \ref{coequiv}.

We construct, in Section \ref{secEG}, examples of tight functions that have both principal and conformal pieces and smooth interfaces between them. Some of these examples are not differentiable.

\subsection{Alternative notions of tight}

\subsubsection{History of the scalar case} \label{s.scalarhistory}

Aronsson proved in 1967 that if $X$ is the closure of a bounded domain in $\R^n$, then a smooth function from $X$ to $\R$ is an AML extension of its values on $\partial X$ if and only if it is infinity harmonic.  At the time, the existence and uniqueness of AML extensions was unknown in general, and it was also unclear what the correct notion of a non-smooth infinity harmonic function should be.

Crandall and Lions remedied the latter problem by introducing the notion of viscosity solution in 1983 \cite{MR690039}.  The definition of viscosity solution was based on a monoticity requirement: roughly speaking, a function is a viscosity solution if whenever it is less (greater) than a smooth solution on the boundary of a domain, it is less (greater) than the smooth solution in the interior as well.  Using this definition, Jensen established the existence and uniqueness of infinity harmonic extensions in 1993 \cite{MR0217665, MR1218686}.

Restricting the class of smooth test functions to {\em cones}, one obtains an analogue of the viscosity solution property called {\em comparison with cones} that easily extends to other metric spaces. Jensen~\cite{MR1218686} proved that infinity harmonic functions satisfy comparison with cones and Crandall, Evans, and Gariepy \cite{MR1861094} proved that a function on a bounded open set $U \subseteq \R^n$ is AML if and only if it satisfies comparison with cones. Champion and De Pascale \cite{MR2341302} adapted this definition
to length spaces, where cones are replaced by functions of the form $\phi(x)=b\,d(x,z)+c$ where $b>0$.  The existence of AML extensions was extended to separable length space domains \cite{MR1884349}, and  \cite{MR2449057} established existence and uniqueness for general length spaces (using game-theoretic techniques).

\subsubsection{Proper definitions for the vector case}

The success of the theory in the case $Z=\R$ relies heavily on monotonicity techniques.  These techniques do not seem to generalize to $Z = \R^m$ with $m > 1$.  As a result, we are unable to generalize all of the $Z = \R$ results to $Z=\R^m$.  This paper goes about as far with the $m>1$ theory as Aronsson went with the $m=1$ theory in the 1960's.  We introduce a notion of optimality (tightness) and describe the smooth functions with that property.  It remains to be seen whether our definition has to be modified in order to establish a satisfactory existence and uniqueness theory.

The following alternative notion of tight was suggested by Robert Jensen. If $U \subseteq \R^n$ is bounded open and $u, v : U \to \R^m$ are Lipschitz, then we say $u$ is {\em measure-tighter} than $v$ if $u = v$ on $\partial U$ and there is an $s > 0$ such that
\begin{equation*}
| \{ |Du| \geq s \}| < | \{ |Dv| \geq s \}|
\end{equation*}
and
\begin{equation*}
| \{ |Du| \geq t \}| \leq | \{ |Dv| \geq t \}|  \quad \mbox{for every } t \geq s.
\end{equation*}
Observe that measure-tight implies tight: this follows from the fact that $u$ tighter than $v$ implies that $u$ is measure-tighter than $v$, which can be deduced from the following:
\begin{equation*}
| B(x,\delta_1) \cap \{ \Lu u > \Lu u(x) - \delta_2 \} | > 0,
\end{equation*}
for every $\delta_1, \delta_2 > 0$ and $x \in U$.

One could also adopt an even stronger notion of optimality and say that $u$ is tighter than $v$ if the $L^p$ norm of $|Du|$ is less than that of $|Dv|$ for all sufficiently large $p$.

Our philosophical objection to both of these definitions is that they depend on the structure of Lebesgue measure, and it is not clear whether a different measure would yield a substantially different definition.  One would hope that the proper notion of optimal extension would depend only on metric space properties of the domain (as in the case when $m=1$) and not on any extraneous structure such as a measure.  Nonetheless, it would certainly be interesting if existence or uniqueness could be established using these alternative definitions.  A different measure theoretic approach in the case that $Z = \R$ and $X$ is a more general measurable length space appears in \cite{MR2242966} (which explores the idea of replacing the Lipschitz norm in the definition of AMLE with the essential supremum of the local Lipschitz norm).

\subsection{Notation and preliminaries}

If $U \subseteq \R^n$, then $C^k(U, \R^m)$ denotes the space of $k$-times continuously differentiable functions from $U$ to $\R^m$. If $u = (u^1, ..., u^m) \in C^1(U, \R^m)$ and $x \in U$, then $Du(x)$ denotes the $m \times n$ matrix with $i,j$-th entry $u^i_{x_j}(x)$. When $u \in C^1(U, \R^m)$, $a \in C(U, \R^n)$, and $|a| = 1$, we write $u_a := u_{x_i} a^i$ and $u_{aa} := u_{x_i x_j} a^i a^j$. Note that $u_{aa} \neq (u_a)_a$ in general.

Aside from the definition of $\Lu u$ above, the most important convention in this article is our choice of matrix norm. Given an $n \times n$ symmetric matrix $A$, we let
\begin{equation*}
\lambda_1(A) \geq \cdots \geq \lambda_n(A),
\end{equation*}
denote the eigenvalues of $A$ listed in non-increasing order. Our norm \EQ{defNorm} can then be written
\begin{equation*}
|A|^2 = \max_{|a| = 1} |A a|^2 = \lambda_1(A^t A).
\end{equation*}
Recall that we choose this norm because it satisfies
\begin{equation*}
\Lu u(x) = |Du(x)|,
\end{equation*}
whenever $U \subseteq \R^n$ is open, $u \in C^1(U,\R^m)$, and $x \in U$. We stress again that this norm is {\em not} induced by the usual matrix inner product $\langle A, B \rangle := \trace(A^t B)$. Moreover, the map $A \mapsto |A|$ is not smooth.

When $\lambda_1(A^t A) > \lambda_2(A^t A)$, there is a unit vector $a$ spanning the principal eigenspace of $A^t A$. In this case, the map $A \mapsto |A|$ is smooth in a neighborhood of $A$ and satisfies
\begin{equation}
\label{dnorm}
|A + s B|^2 = |A|^2 + 2 s (A a)^t (Ba) + O(s^2),
\end{equation}
for any $m \times n$ matrix $B$ and $s \in \R$. Moreover, the constant in the $O(s^2)$ term depends continuously on $|A|$, $|B|$, and $(\lambda_1(A^t A) - \lambda_2(A^t A))^{-1}$.

\subsection{Acknowledgements}

The first co-author has collaborated with Assaf Naor on the related project of showing that Lipschitz functions from subsets of length spaces to {\em trees} have unique optimal Lipschitz extensions \cite{naorsheffield}. The work with Naor also contains a general definition of discrete infinity harmonic (similar to the one we use here), an existence result for discrete infinity harmonic extensions, and a one-triangle version of the counterexample in \SEC{secGraph}.

We also thank Robert Jensen for many helpful discussions.


\section{Tight functions on graphs}
\label{secGraph}

\subsection{The finite case}

As stated in the introduction, discrete infinity harmonic extensions are not unique in the vector-valued case. To see this, consider the two distinct discrete infinity harmonic embeddings of a finite graph into $\R^2$ displayed in Figure \ref{nonuniq}.
\begin{figure}[h]
\caption{Two distinct discrete infinity harmonic embeddings of a finite graph into $\R^2$.}
\label{nonuniq}
\bigskip
\includegraphics[width=3in]{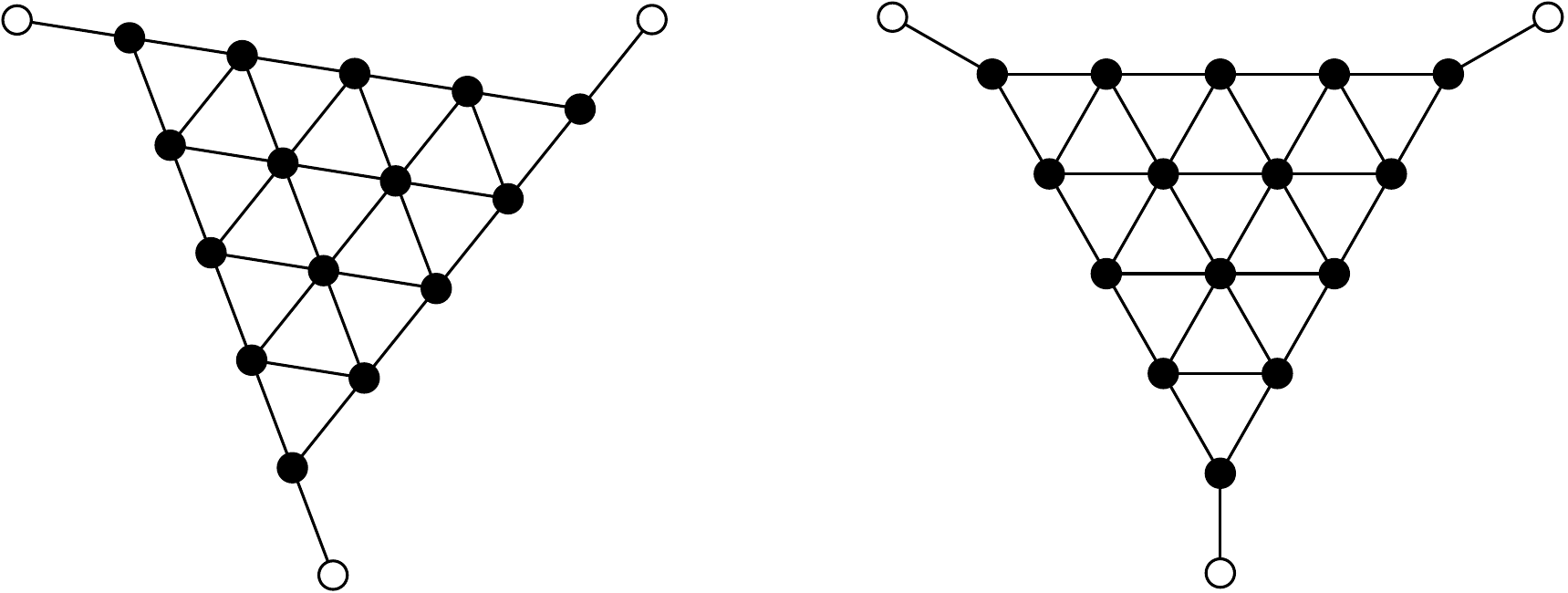}
\end{figure}
Here the outer three vertices are fixed and we are free to select the positions of the inner vertices. The graph on the right has a smaller overall Lipschitz constant.

In fact, this example demonstrates more than just a failure of uniqueness. One can show that if we fix any additional  vertex $x$ on the boundary of the inner triangle in the left hand graph, then the extension is the unique one with the given Lipschitz constant.  (Once we fix such a vertex $x$, the corner vertex colinear with $x$ and a boundary point becomes fixed. From this, it follows that the other two corners are fixed, etc.) Thus any algorithm that identifies an extension as having a sub-optimal Lipschitz constant must consider many vertices simultaneously. In particular, the non-local nature of the definition of tighter in \EQ{defTightG} is necessary.

In the scalar case, discrete infinity harmonic (hence tight) extensions can be computed using the following procedure (see \cite{MR1685133} and \cite{MR2449057}):
\begin{quote}
Find a path $x_0, ..., x_n$ in $G$ with $x_0, x_n \in Y$, $x_1, ..., x_{n-1} \in X \setminus Y$, and $n > 1$ such that $(u(x_n) - u(x_0))/n$ is as large as possible. Extend $g$ to $Y \cup \{ x_i \}$ by setting $g(x_i) = (1 - i/n) u(x_0) + (i/n) u(x_n)$ and then set $Y := Y \cup \{ x_i \}$. Repeat until $X = Y$.
\end{quote}
While we do not know of such a simple algorithm in the vector-valued case, the existence and uniqueness proof below is based on the idea that tight functions can be determined ``steepest part first.''

\begin{proof}[{\bf Proof of \THM{EUG}}]
Given any extension $u : X \to \R^m$ of $g$, let $lv(u) \in \R^{|X \setminus Y|}$ be the values of $\Lg u$ listed in non-increasing order. Observe that if $u$ is tighter than some $v : X \to \R^m$, then $lv(u) < lv(v)$ in the lexicographical ordering on $\R^{|X \setminus Y|}$. Thus, if $lv(u)$ is lexicographically minimal among all extensions $u : X \to \R^m$ of $g$, then $u$ is a tight extension of $g$.

To find a lexicographically minimal extension of $g$, simply observe that the set of obtainable $lv(u)$  (as $u$ ranges over all possible extensions) is an unbounded closed polyhedron in the subset of $[0,\infty)^{|X \setminus Y|}$ consisting of points with non-increasing coordinates.  An induction on the dimension $|X \setminus Y|$ shows that that the lexicographic minimum is obtained on this set.  (The set of points minimizing the first coordinate is bounded and closed and non-empty, the subset of this set minimizing the second coordinate is thus bounded and closed and non-empty, etc.)  Thus, there indeed exists a $u$ for which $lv(u)$ is lexicographically minimal, and this $u$ is tight.


Now consider an arbitrary extension $v$ not equal to $u$.  We claim that $u$ is tighter than $v$ (and in particular that $u$ is the unique tight extension). For contradiction, suppose $u$ is not tighter than $v$.  Since $v$ is also not tighter than $u$, we must have
\begin{equation*}
K := \max \{ \Lg u : \Lg u \neq \Lg v \} = \max \{ \Lg v :  \Lg u \neq \Lg v \},
\end{equation*}
where we define $K = 0$ if $\Lg u \equiv \Lg v$.

Set $w := (u + v) / 2$ and observe that
\begin{equation}
\label{edgeconvex}
|w(x) - w(y)| \leq \frac{1}{2} |u(x) - u(y)| + \frac{1}{2} |v(x) - v(y)|,
\end{equation}
whenever $x \sim_E y$. In particular, if $Su(x) \geq K$, then
\begin{equation*}
Sw(x) \leq \frac{Su(x) + Sv(x)}{2} \leq Su(x).
\end{equation*}
Since $w$ is not tighter than $u$, we must have $\Lg w = \Lg u$ on $\{ \Lg u \geq K \}$.

Consider the set of edges
\begin{equation*}
\tilde E := \{ (x,y) \in E : |w(x) - w(y)| \geq K \},
\end{equation*}
of length at least $K$. Since equality holds in \eqref{edgeconvex} only if $u(x) - u(y) = v(x) - v(y)$, we conclude that $u(x) - u(y) = v(x) - v(y)$ whenever $x \sim_{\tilde E} y$. Therefore $u(x) = v(x)$ whenever $x \in X$ is connected to $Y$ via edges in $\tilde E$.

If every vertex in $\{ \Lg w \geq K \}$ is connected to $Y$ via edges in $\tilde E$ then we may conclude that $u \equiv v$ on $\{ \Lg w \geq K \} \supseteq \{ \Lg u \geq K \}$. In particular, $\Lg u = \Lg v$ on $\{ \Lg u \geq K \}$. Since $v$ is not tighter than $u$, we must also have $\Lg u = \Lg v$ on $\{ \Lg v \geq K \}$. Thus $K = 0$ and $u \equiv v$. Since we assumed $u \not \equiv v$, there must be a non-empty maximal $\tilde E$-connected set $X_0 \subseteq \{ \Lg w \geq K \}$.

For small $\ep > 0$, define $w_\ep : X \to \R^m$ by
\begin{equation*}
w_\ep(x) := \begin{cases}
w(x) & \mbox{if } x \in X \setminus X_0, \\
(1 - \ep) w(x) & \mbox{if } x \in X_0.
\end{cases}
\end{equation*}
Since $|w(x) - w(y)| < K$ every $(x,y) \in E \setminus \tilde E$, we may select a small $\ep > 0$ such that $|w_\ep(x) - w_\ep(y)| < K$ for $(x,y) \in E \setminus \tilde E$. Since $X_0$ has no outgoing edges in $\tilde E$, we also have $|w_\ep(x) - w_\ep(y)| \leq |w(x) - w(y)|$ if $(x,y) \in \tilde E$ and $Sw_\ep(x) = (1 - \ep) Sw(x)$ if $x \in X_0$. In particular,
\begin{equation*}
\max \{ \Lg w_\ep : \Lg w_\ep \neq \Lg u \} < \max \{ \Lg u : \Lg w_\ep \neq \Lg u \},
\end{equation*}
contradicting the fact that $u$ is tight.
\end{proof}

One might expect the convergence of the $p$-harmonic extensions to be corollary of the following easy fact: If $v$ is tighter than $u$, then
\begin{equation*}
\sum_{x \in X \setminus Y} (\Lg u(x))^p > \sum_{x \in X \setminus Y} (\Lg v(x))^p,
\end{equation*}
for all sufficiently large $p > 0$. However, it seems a more subtle argument is required.

\begin{proof}[{\bf Proof of \THM{PlimG}}]
Since the $u_p$ are bounded as $p \to \infty$, we know that any sequence $p_k \to \infty$ has a further subsequence along which $u_p$ converges to some function $u : X \to \R^m$.  Suppose for contradiction that $v : X \to \R^m$ is the unique tight extension of $g$ and $v \not \equiv u$. Let
\begin{equation*}
K := \max \{ \Lg u : \Lg v \neq \Lg u \} > \max \{ \Lg v : \Lg v \neq \Lg u \},
\end{equation*}
and
\begin{equation*}
v_k(x) := \left\{ \begin{array}{ll}
u_{p_k}(x) & \mrbox{if } \Lg u(x) > K, \\
v(x) & \mrbox{otherwise.}
\end{array} \right.
\end{equation*}
We show that $v_k$ has smaller $p_k$-energy than $u_{p_k}$ for sufficiently large $k$.

Let
\begin{equation*}
Z := Y \cup \{ \Lg u > K \} = Y \cup \{ \Lg v > K \},
\end{equation*}
and observe that $v |_Z$ is tight on $G |_Z$, where $G |_Z$ is the reduct of $G$ to the vertex set $Z$. Since $\Lg u = \Lg v$ on $Z$, the proof of \THM{EUG} implies that $u$ must also be tight on $G |_Z$. In particular, $u = v$ on $Z$ and $v_k \to v$ as $k \to \infty$.

Observe that
\begin{equation*}
\Lg v_k = \Lg u_{p_k} \quad \mbox{on } \{ \Lg u > K \},
\end{equation*}
for all sufficiently large $k$. Moreover, if we define
\begin{equation*}
\delta := \max \{ Su : Sv \neq Su \} - \max \{ Sv : Sv \neq Su \},
\end{equation*}
then
\begin{equation*}
\Lg v_k < K - \delta / 2 \quad \mbox{on } \{ \Lg u \leq K \},
\end{equation*}
and
\begin{equation*}
\Lg u_{p_k} \geq K - \delta / 4 \quad \mbox{ on } \{ \Lg u = K \},
\end{equation*}
 for all sufficiently large $k$. Thus we compute
\begin{align*}
\sum_{X \setminus Y} (\Lg u_{p_k})^{p_k} - \sum_{X \setminus Y} (\Lg v_k(x))^{p_k}
& \geq \sum_{\{ \Lg u = K \}} (\Lg u_{p_k})^{p_k} - \sum_{\{ \Lg u \leq K \}} (\Lg v_k(x))^{p_k} \\
& \geq (K - \delta/4)^{p_k} - |X \setminus Y| (K - \delta/2)^{p_k} \\
& > 0,
\end{align*}
for all sufficiently large $k$.
\end{proof}

\subsection{Non-uniqueness of tight extensions on infinite graphs}

Recall if $G = (X, E, Y)$ is finite and $u, v : X \to \R$ are tight, then
\begin{equation}
\label{linearcmp}
\max_X |u - v| = \max_Y |u - v|.
\end{equation}
When $u, v : X \to \R^m$ and $m > 1$, this fails. Indeed, consider the two embeddings of a four-vertex star-shaped graph into $\R^2$ displayed in Figure \ref{gadget1}.
\begin{figure}[h]
\caption{Two embeddings of a star-shaped four-vertex graph into $\R^2$.}
\label{gadget1}
\includegraphics[width=2in]{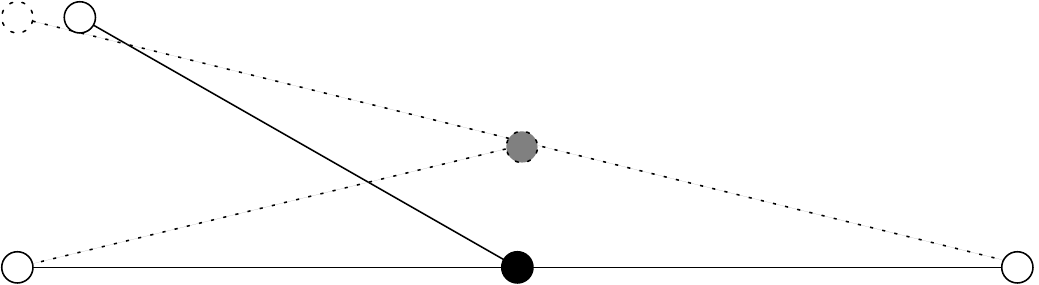}
\end{figure}
In this example, the hub of the star is an interior vertex while the three points of the star are boundary vertices. The two fixed boundary vertices are mapped to $(1,0)$ and $(-1, 0)$. The free boundary vertex is mapped to $(-(1 - \ep^2)^{1/2}, \ep)$ or $(-1, \ep)$ and the interior vertex is then mapped to $(0,0)$ or $(0, \ep/2)$, respectively. Since the free boundary vertex moves by only
\begin{equation*}
r - (r^2 - \ep^2)^{1/2} = \frac{\ep^2}{2r} + O(\ep^4),
\end{equation*}
we see that \eqref{linearcmp} fails in the vector-valued case.

We call this example an {\em amplifier gadget}, as it amplifies certain perturbations of the boundary data. We think of the free boundary vertex as the {\em input} and the interior vertex as the {\em output}. To obtain our non-uniqueness example, we will
chain amplifier gadgets together.  That is, we will start with a single amplifier gadget and then inductively attach the output of a new gadget to the input of the existing compound gadget (see Figure \ref{gadget2}).  As the number of gadgets in the chain of gadgets gets longer, the size of the input change to the last gadget needed to affect a given output change in the first gadget will tend to zero.  (Recall that we have existence and uniqueness for finite graphs.)  Once we extend the chain to infinite length, we can construct two tight functions $u$ and $v$ from the graph vertices to $\R^2$ which disagree on the first gadget in the chain, but whose difference tends to zero as one moves further up the chain.  The values $Su$ and $Sv$ will increase (converging to some limit) as one goes further up the chain.  We will have to prove that the $u$ and $v$ we construct are tight: it will be obvious from construction and our finite graph results that the functions restricted to the first $n$ gadgets in the chain are tight --- that is, one cannot construct a tighter function than $u$ (the argument for $v$ is similar) by modifying $u$ on a finite set of interior vertices.  However, we will also show that changing $u$ on an infinite set must make $Su$ larger on an infinite set (including places where $Su$ is arbitrarily close to its maximum).

To be precise, we consider the infinite graph $G = (X, E, Y)$ with vertices
\begin{equation*}
X := \{ a_i \} \cup \{ b_i \} \cup \{ c_i \},
\end{equation*}
edges
\begin{equation*}
E := \{ ( a_i, c_i) \} \cup \{ (b_i, c_i) \} \cup \{ (c_i, c_{i+1} ) \},
\end{equation*}
and boundary vertices
\begin{equation*}
Y := \{ a_i \} \cup \{ b_i \},
\end{equation*}
where $i$ ranges over the natural numbers $\mathbb{N}$. We let $r_0 := 1$ and $\ep_0 = 1/4$ and then inductively select $(r_i, \ep_i)$ such that
\begin{equation}
\label{ridef}
r_{i+1} = (r_i^2 + (\ep_i/2)^2)^{1/2},
\end{equation}
and
\begin{equation}
\label{epidef}
\ep_{i+1} / 2 = r_i - (r_i^2 - \ep_i^2)^{1/2}.
\end{equation}
We let $g : Y \to \R^2$ be the unique map such that $g(a_0) = (-1,0)$, $g(b_0) = (1,0)$,
\begin{equation}
\label{abrotate}
\frac{g(a_{i+1}) - g(b_{i+1})}{2r_{i+1}} = \left[ \begin{matrix}0 & 1/\sqrt 2 \\ -1/\sqrt 2 & 0 \end{matrix} \right] \frac{g(a_i) - g(b_i)}{2r_i},
\end{equation}
and
\begin{equation}
\label{abshift}
\frac{g(a_{i+1}) + g(b_{i+1})}{2} = g(a_i) + \ep_i \left[ \begin{matrix}0 & 1/\sqrt 2 \\ -1/\sqrt 2 & 0 \end{matrix} \right] \frac{g(a_i) - g(b_i)}{2r_i}.
\end{equation}
Finally, we define two extensions $u, v : X \to \R^2$ of $g$ by
\begin{equation*}
u(c_i) := \begin{cases}
\frac{1}{2}(g(a_i) + g(b_i)) & \mbox{if $i$ even}, \\
\frac{1}{4}(g(a_{i+1}) + g(b_{i+1})+ 2 g(b_i)) & \mbox{if $i$ odd},
\end{cases}
\end{equation*}
and
\begin{equation*}
v(c_i) := \begin{cases}
\frac{1}{2}(g(a_i) + g(b_i)) & \mbox{if $i$ odd}, \\
\frac{1}{4}(g(a_{i+1}) + g(b_{i+1})+ 2 g(b_i)) & \mbox{if $i$ even}.
\end{cases}
\end{equation*}
See Figure \ref{gadget2} for a diagram of these embeddings.

\begin{figure}[h]
\caption{The restriction of the infinite graph example to the vertices $\{ a_i, b_i, c_i, a_{i+1}, b_{i+1}, c_{i+1}, c_{i+2} \}$ in the case $i$ is even. Here we have rotated the embedding by $\pi i/2$.}
\label{gadget2}
\bigskip
\includegraphics[width=3in]{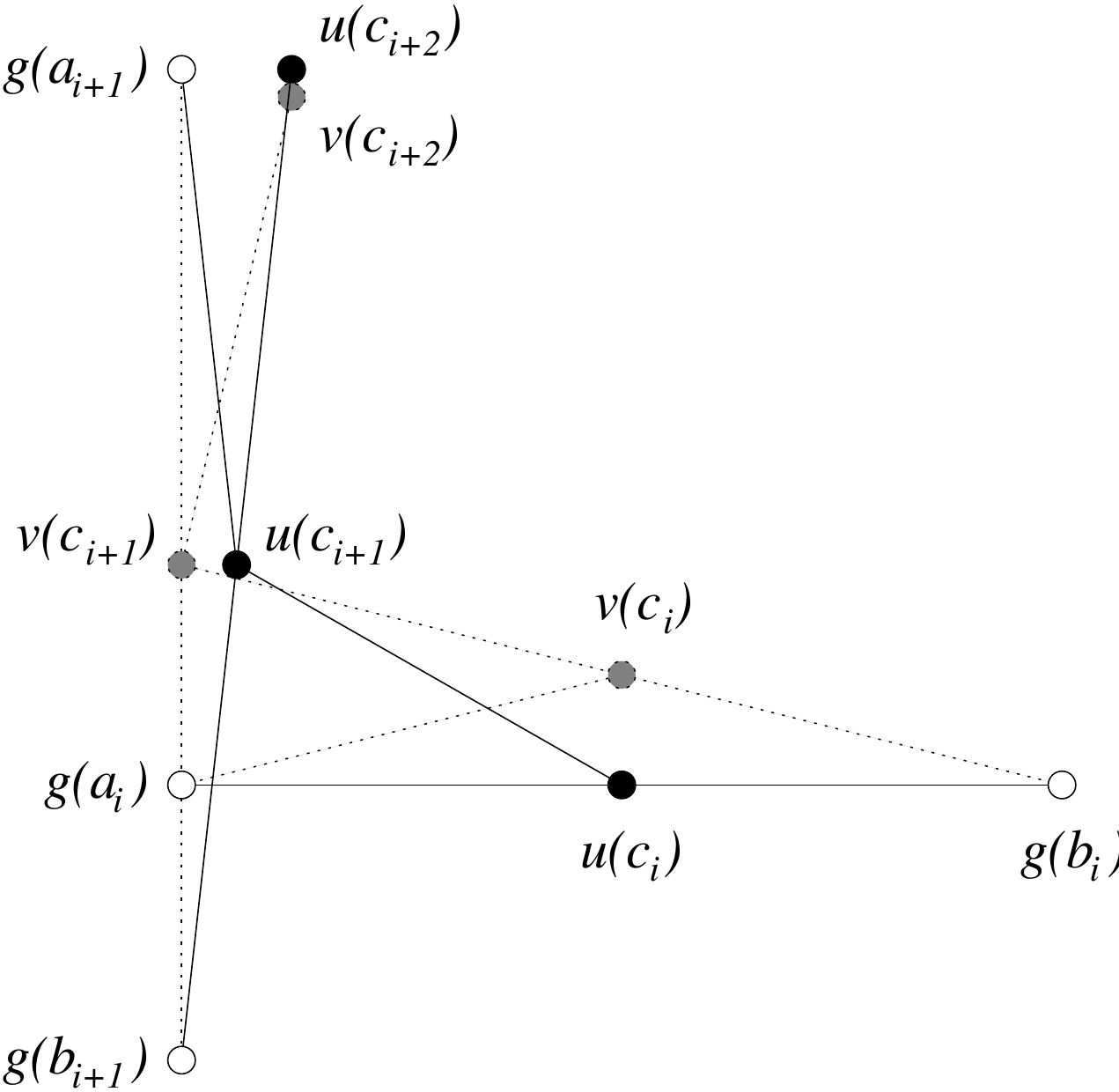}
\end{figure}

\begin{proposition}
\label{nonunique}
The function $g : Y \to \R^2$ is bounded and both $u, v : X \to \R^2$ are tight on $G$. In particular, tight extensions on infinite graphs are not necessarily unique even if the boundary data is bounded.
\end{proposition}

\begin{proof}
We first check that $g$ is bounded. From \eqref{ridef} it follows that $i \mapsto r_i$ is increasing. Together with \eqref{epidef}, this implies that $\ep_{i+1} \leq 2 \ep_i^2$. Since $\ep_0 = 1/4$, it follows that $\ep_i \to 0$ as $i \to \infty$ and $\sum_{i=0}^\infty \ep_i < \infty$. Since \eqref{ridef} implies $r_{i+1} \leq r_i + \ep_i$, it then follows that $r_i \to \bar r < \infty$ as $i \to \infty$.

Now consider how the points $g(a_i)$ are generated in \eqref{abrotate} and \eqref{abshift}. During each iteration, the process rotates by $\pi/2$ radians clockwise. In particular, if we take four steps at once, we obtain
\begin{equation*}
|g(a_{i+4}) - g(a_i)| = O(|r_{i+4} - r_i|) + O(|\ep_{i+4} - \ep_i|) = O(\ep_i).
\end{equation*}
Since $\sum_{i=0}^\infty \ep_i < \infty$, we see that $|g(a_{4i}) - g(a_0)|$ is bounded uniformly in $i$. This implies that $g$ is bounded.

Next we check that $u$ is tight. Suppose instead that $w : X \to \R^2$ is tighter than $u$. Observe that our choice of $r_{i+1}$ in \eqref{ridef} guarantees that $i \mapsto \Lg u(c_i)$ is nondecreasing. Indeed, if $i$ is even then
\begin{equation*}
\Lg u(c_i) = r_i,
\end{equation*}
and if $i$ is odd then
\begin{equation*}
\Lg u(c_i) = (r_i^2 + (\ep_i/2)^2)^{1/2} = r_{i+1} > r_i.
\end{equation*}

When $i$ is even we have
\begin{equation*}
u(c_i) = \frac{1}{2} (g(a_i) + g(b_i)),
\end{equation*}
and
\begin{equation*}
\Lg u(c_i) = |u(c_i) - g(a_i)| = |u(c_i) - g(b_i)|.
\end{equation*}
Thus if $w(c_i) \neq u(c_i)$ we have
\begin{equation*}
\Lg w(c_i) > \Lg u(c_i) = \max_{j \leq i} \Lg u(c_j).
\end{equation*}
Since $w$ is tighter than $u$, we must have $w(c_i) = u(c_i)$ for all large even $i$.

Given that $i > 1$ is even and $u(c_i) = w(c_i)$, observe that
\begin{equation*}
u(c_{i-1}) = \frac{1}{2} (u(c_i) + g(b_i)),
\end{equation*}
and
\begin{equation*}
\Lg u(c_{i-1}) = |u(c_{i-1}) - g(b_i)| = |u(c_{i-1}) - u(c_i)|.
\end{equation*}
In particular, if $w(c_{i-1}) \neq u(c_{i-1})$, then
\begin{equation*}
\Lg w(c_{i-1}) > \Lg u(c_{i-1}) = \max_{j \leq i - 1} \Lg u(c_j).
\end{equation*}
Thus we conclude that $w(c_i) = u(c_i)$ for all sufficiently large $i$.

Now, let $i$ be the largest integer such that $w(c_i) \neq u(c_i)$. The arguments we just gave immediately give
\begin{equation*}
\Lg w(c_i) > \Lg u(c_i) = \max_{j \leq i} \Lg u(c_j),
\end{equation*}
and thus $w$ can not be tighter than $u$. It follows that $u$ is tight.

The argument that $v$ is tight is the same except that we switch odd and even.
\end{proof}

\begin{remark}
The definition, existence, and uniqueness of tight extensions easily generalizes to the case of weighted graphs. That is, when we have a function $d : E \to (0, \infty)$ and define
\begin{equation*}
\Lg u(x) := \max_{y \sim_E x} \frac{|u(y) - u(x)|}{d(y,x)}.
\end{equation*}
With this generalization, it is possible modify our non-uniqueness example so that the graph has finite diameter. That is, we can set
\begin{equation*}
d(a_i,c_i) = d(b_i, c_i) =  d(c_i, c_{i+1}) = \frac{1}{2^i},
\end{equation*}
and then reduce the size of the amplifier gadgets so that the boundary data is still bounded and Lipschitz.

We thus obtain an example of a finite diameter and connected length space $X$, a non-empty subset $Y \subseteq X$, and a Lipschitz function $g : Y \to \R^2$ such that $g$ has at least two tight extensions to $X$.
\end{remark}

\subsection{An edge-based treatment of tight extensions}

We finish this section by providing an alternative characterization of tight extensions on graphs. Given $u : X \to \R^m$ and $e = \{ x, y \} \in E \setminus Y^2$, we define $\Le u(e) := |u(x) - u(y)|$. If two functions $u, v : X \to \R^m$ agree on $Y$ and satisfy
\begin{equation*}
\sup \{ \Le u : \Le u < \Le v \} > \sup \{ \Le v : \Le u > \Le v \},
\end{equation*}
we say that $v$ is {\em edge-tighter} on $G$. We say that $u$ is {\em edge-tight} if there is no edge-tighter $v$.

\begin{proposition}
\label{edgeTight}
Let $G = (X, E, Y)$ denote a connected finite graph with vertex set $X$, edge set $E$, and a distinguished non-empty set of vertices $Y \subseteq X$. A function $u : X \to \R^m$ is tight on $G$ if and only if it is edge-tight on $G$.
\end{proposition}

\begin{proof}
We first observe that the proof of \THM{EUG} also works for edge-tight extensions. With existence and uniqueness of edge-tight extensions in hand, it is enough to prove that tight implies edge-tight.

Suppose $u : X \to \R^m$ is tight on $G$ and $v : X \to \R^m$ is any other extension of $u |_Y$. \THM{EUG} implies that $u$ is necessarily tighter than $v$. Let
\begin{equation*}
K := \max \{ \Lg v : \Lg u \neq \Lg v \} > \max \{ \Lg u : \Lg u \neq \Lg v \}.
\end{equation*}
Since both $u$ and $v$ are tight on $Z := \{ \Lg v > K\} = \{ \Lg u > K \}$, it follows that $u = v$ on $Z$. This implies that
\begin{equation*}
\max \{ \Le u : \Le v < \Le u \} \leq K \leq \max \{ \Le v : \Le v > \Le u \}.
\end{equation*}
In particular $v$ is not edge-tighter than $u$. Thus tight implies edge-tight.
\end{proof}


\section{Tight functions on $\R^n$}
\label{secTight}

\subsection{Non-uniqueness of AML extensions} We showed in \SEC{secGraph} that discrete infinity harmonic extensions on a graph are not necessarily unique. In higher dimensions the continuum criterion \EQ{defAML} similarly fails to characterize a unique extension.

\begin{example}
For $t \in [0, 1]$, let $u_t : B_1(0) \subseteq \R^2 \to \R^2$ be given by
\begin{equation}
\label{egball}
u_t(z) := t z^2 + (1 - t) z^2 / |z|,
\end{equation}
and observe that
\begin{equation*}
\Lu u_t(x) = 2 + 2 t (|x| - 1).
\end{equation*}
If follows that if $V \subseteq B_1(0)$ and $x \in \partial V$ maximizes $x \mapsto |x|$, then $\sup_V \Lu u_t = \Lu u_t(x)$; we may find distinct $x_1,x_2$ on $\partial V$, arbitrarily near $x$, satisfying $|x_1|=|x_2|$, so $\Lip_{\partial V}(u_t) = \Lip_V(u_t) =  2 + 2 t (|x| - 1)$. Thus each $u_t$ satisfies \EQ{defAML}. Note also that the $u_t$ agree on the boundary of $\partial B_1(0)$ and $t \mapsto \Lu u_t(x)$ is decreasing for all $x \in B_1(0)$. Thus $u_{t_1}$ is tighter than $u_{t_2}$ on $B_1(0)$ when $t_1 > t_2$.
\end{example}

It is interesting to note that the PDE \EQ{pdePr} detects the non-optimality of the $u_t$ when $t < 1$. Indeed, when $t < 1$ the unit vector field
\begin{equation*}
a(z) := i z / |z|,
\end{equation*}
is a principal direction field for $u_t$ on $B_1(0) \setminus \{ 0 \}$. Since the $u_t$-image of a principal flow curve is not a line, we see that $u_t$ does not satisfy \EQ{pdePr}.

\subsection{Limiting Euler-Lagrange equations}

To derive the PDE \EQ{pdePr} as the limit of the Euler-Lagrange equations for \EQ{penergy}, recall (e.g., Chapter 8, Theorem 6 of \cite{MR1625845}) that if $u \in C^2(U, \R^m)$ minimizes a functional of the form
\begin{equation*}
v \mapsto \int_U F(Du) \,dx,
\end{equation*}
where $F \in C^2(\R^{m \times n})$, then $u$ satisfies
\begin{equation*}
- (F_{p_{k,i}}(Du))_{x_i} = 0 \mrbox{in} U,
\end{equation*}
where $F_{p_{k,i}}$ denotes the derivative of $F$ with respect to its $ki$-th input.

If $u \in C^2(U)$ minimizes \EQ{penergy} and has a principal direction field $a \in C^1(U, \R^n)$, then $Du(U)$ is contained in the open set where the norm $A \mapsto |A|$ is smooth. Using \EQ{dnorm}, we conclude that
\begin{equation*}
- (p |Du|^{p-2} u_{x_j} a^j a^k)_{x_k} = 0 \mrbox{in} U.
\end{equation*}
Distributing the derivative with respect to $x_k$ and normalizing, this is equivalent to
\begin{equation*}
- |u_a|^{-2} (u_{aa} \cdot u_a) u_a - \frac{1}{p-2} ((u_a)_a + u_a \div (a)) = 0 \mrbox{in} U.
\end{equation*}
Here we have used the shorthand $u_a := u_{x_i} a^i$ and $u_{aa} := u_{x_i x_j} a^i a^j$ introduced in the preliminaries. Since $u_{aa} \cdot u_a = (u_a)_a \cdot u_a$, this system is equivalent to the pair
\begin{equation*}
\left\{ \begin{aligned}
- (|u_a|^{-2} u_a \otimes u_a) (u_a)_a  & = \frac{1}{p-1} u_a \div(a), \\
- (I - |u_a|^{-2} u_a \otimes u_a) (u_a)_a & = 0.
\end{aligned} \right.
\end{equation*}
Sending $p \to \infty$ yields \EQ{pdePr}.

\subsection{Proofs of main results}

\begin{proof}[{\bf Proof of \THM{Pr}}]
($\Rightarrow$) Suppose $u$ is tight. We first claim that
\begin{equation}
\label{nocurv}
(|u_a|^{-1} u_a)_a = 0,
\end{equation}
which implies that the images of the flow curves of $a$ are straight lines (but does not yet imply that the map is length preserving).

Set $b := |u_a|^{-1} u_a$ and observe that, since $b$ is a unit vector, we have $b \cdot b_a = 0$ and therefore $(b \cdot b_a)_a = b \cdot (b_a)_a + |b_a|^2 = 0$. We now show that if \EQ{nocurv} fails (i.e., the images of the flow curves are not straight lines), then we may modify $u$ (via a smooth perturbation that partially straightens these lines) to produce a tighter function.  Assuming \EQ{nocurv} fails, we may rescale and translate to obtain that $(b_a)_a \cdot b < 0$ on $\bar B_1(0) \subseteq U$.

Define the standard bump function
\begin{equation*}
\varphi(x) := \begin{cases} e^{- (1 - x^2)^{-2}} & \mbox{if } |x| < 1, \\
0 & \mbox{if }|x| \geq 1, \end{cases}
\end{equation*}
and set
\begin{equation*}
v := u + s \varphi b_a,
\end{equation*}
for some $s > 0$ to be determined. Using $u_a \cdot b_a = 0$ and \EQ{dnorm}, compute
\begin{align*}
|Dv|^2 & = |Du|^2 + 2 s u_a \cdot (\varphi b_a)_a + O(s^2) \\
& = |Du|^2 + 2 s \varphi u_a \cdot (b_a)_a + O(s^2).
\end{align*}
Recall that the constant in the $O(s^2)$ depends continuously on $Du$, $D(\varphi b_a)$, and $(\lambda_1(Du^t Du) - \lambda_2(Du^t Du))^{-1}$. Since all three of these are bounded on $\bar B_1(0)$, we may select a constant for the $O(s^2)$ term that works on all of $\bar B_1(0)$. Since $u_a \cdot b_a < 0$ on $\bar B_1(0)$, we see that $v$ is tighter than $u$ for all small $s > 0$.

Thus if \EQ{pdePr} fails, we may assume by \EQ{nocurv} that $u_a \cdot (u_a)_a > 0$ on $\bar B_1(0) \subseteq U$. (If instead $u_a \cdot (u_a)_a < 0$ we can replace $a$ with $-a$.) Since $D |u_a|^2 \neq 0$ in $B_1(0)$ and $u \in C^3(U, \R^m)$, we may assume $|u_a|^2$ has a strict maximum on $\bar B_1(0)$ at some point $x^* \in \partial B_1(0)$. Thus $D|u_a|^2(x^*)$ is a positive multiple of the outward unit normal $a^*$ to $\partial B_1(0)$ at $x^*$. Since $a \cdot D|u_a|^2 = 2 u_a \cdot (u_a)_a > 0$, we conclude that $a(x^*) \cdot a^* > 0$. In particular, $\varphi_a(x) / \varphi(x) \to - \infty$ as $x \to x^*$.

Set
\begin{equation*}
v := u + s \varphi u_a,
\end{equation*}
and use $u_a \cdot (u_a)_a = 0$ and \eqref{dnorm} to compute
\begin{align*}
|D v|^2 & = |Du|^2 + 2 s u_a \cdot (s \varphi u_a)_a + O(s^2) \\
& = |Du|^2 + 2s ( \varphi_a |u_a|^2 + \varphi u_a \cdot (u_a)_a) + O(s^2).
\end{align*}
As before, there is a constant for the $O(s^2)$ term that works for all of $\bar B_1(0)$. Since $\varphi_a(x)/\varphi(x) \to - \infty$ as $x \to x^*$, we know that the linear term in this Taylor expansion is negative in a neighborhood of $x^*$. Since $x^*$ is a strict maximum for $|u_a| = |\Lu u|$ on $\bar B_1(0)$, it follows that $v$ is tighter than $u$ when $s > 0$ is small.

($\Leftarrow$) Suppose $u$ satisfies \EQ{pdePr} in $U$ and that $v$ is tighter than $u$. Let $x \in U$ be a point where
\begin{equation}
\label{xbetter}
\Lu u(x) > \Lu v(x) > \max \{ \Lu v : \Lu v > \Lu u \}.
\end{equation}
Let $\gamma : [0, T] \to \bar U$ be a unit-speed parameterization of the principal flow curve through $x$. It follows from \EQ{pdePr} that $u \circ \gamma$ is affine. We also have
\begin{equation*}
(v \circ \gamma)'(t) \leq |Dv(\gamma(t))| \leq |Du(\gamma(t))| = (u \circ \gamma)'(t),
\end{equation*}
for all $t \in (0, T)$. Since $(v \circ \gamma)'(t_0) < (u \circ \gamma)'(t_0)$, we have
\begin{equation*}
v(\gamma(T)) - v(\gamma(0)) < u(\gamma(T)) - u(\gamma(0)),
\end{equation*}
contradicting the fact that $u = v$ on $\partial U$.
\end{proof}

\begin{lemma}
\label{coequiv}
Suppose $U \subseteq \R^2$ is bounded and open and $u \in C^\infty(U, \R^2)$ is analytic in a neighborhood of $\bar U$. If $u$ is one-to-one and $u''$ does not vanish on $U$, then \eqref{e.uderivatives} and \eqref{e.delta1} are equivalent.
\end{lemma}

\begin{proof}
To prove the equivalence of \eqref{e.uderivatives} and \eqref{e.delta1}, write $f=\log v'$ and note that since $f$ is analytic the gradient of $\Re f$ is $\overline f'$ (viewed as a vector in $\R^2$). Note also that $\Re f$ has the same argument as $(f')^{-1}$.  Thus, $\Delta_1 \Re(f) \leq 0$ (or equivalently, $\Delta_\infty \Re(f) \geq 0$) if and only if $\Re [f''/(f')^2] \geq 0$.  Plugging in $f = \log v'$, we have
$f' = v''/v'$ and $f'' = (v'''v' - v''v'')/(v')^2$ so that $$\Re \frac{f''}{(f')^2} = \Re \frac{v'''v' - v''v''}{(v'')^2} \geq 0,$$ which is equivalent to $\Re [v'''v' / (v'')^2] \geq 1$. Now let us convert back to a statement about $u$.  We may suppose that $0 \in U$, $u(0) = 0$ and $u'(0) = 1$, at the more general result can be obtained from this case by composing $u$ with affine functions.  Expanding about zero, we see that if
\begin{equation*}
u(z) = z + a_2 z^2 + a_3 z^3 + \cdots,
\end{equation*}
and
\begin{equation*}
v(z) = z + b_2 z^2 + b_3 z^3 + \cdots,
\end{equation*}
then
\begin{equation*}
z = u \circ v = z + (a_2+b_2)z^2 + (a_3 + b_3 + 2a_2b_2)z^3 + \cdots.
\end{equation*}
We conclude that $v'' = -u''$ and $v''' + u''' + 3(v'')^2 = 0$.   Thus, in this case
\begin{equation*}
\Re \frac{v'''v'}{(v'')^2} = \Re \frac{v'''}{(v'')^2}\geq 1 \quad \mbox{if and only if} \quad \Re  \frac{u'u'''}{(u'')^2} = 3 - \Re \frac{v'''}{(v'')^2} \leq 2,
\end{equation*}
 which is precisely \eqref{e.uderivatives}.
\end{proof}

Before providing a proof of \THM{Co}, we consider a simple family of tight conformal maps.

\begin{example}
Suppose $u(z) = z^m$ or more generally $u(z) = e^{m \log z}$ for some fixed branch of $\log$ where $m$ may be complex.  In this case we have identically
\begin{equation*}
u' u'''/(u'')^2 = (m-2)/(m-1) = 1 - 1/(m-1).
\end{equation*}
Thus $\Re [1 - 1/(m-1)] \leq 2$ provided $\Re (m-1)^{-1} \geq -1$, or equivalently $m \not \in B_{1/2}(1/2)$.  In particular, if $m$ is real, then  \eqref{e.uderivatives} states that $m \not \in (0,1)$.

By taking $u(z) = az^m$ and choosing $m$, $a$, and $z$ we obtain a generic set of values in $\C^3$ for the triple
\begin{equation*}
(u'(z),u''(z),u'''(z)) = (a mz^{m-1}, a m(m-1)z^{m-2}, a m(m-1)(m-2)z^{m-3}).
\end{equation*}
Thus these monomials parameterize the space of inputs to \eqref{e.uderivatives}.
\end{example}

\begin{proof}[{\bf Proof of \THM{Co}}]
($\Rightarrow$)  Recall \EQ{e.delta1}, which states (assuming $u$ is one-to-one) that for each $s>0$ the set $\mathcal S := \{ |Du^{-1}| \geq s \}$ is convex in a local sense: each point $x$ in $uU \cap \partial \mathcal S$ has a neighborhood whose intersection with $\mathcal S$ is convex.  If   \EQ{e.delta1} fails, then there is a neighborhood on which $u$ is one-to-one and $\{ |Du^{-1}| \geq s \}$ is not convex for some $s > 0$.  We show that in this case, we can perturb $u$ to create a function tighter than $u$.  By restricting to a subdomain, we may assume that $D w \neq 0$ in $U$, where $w := |Du|^2$. Let
\begin{equation*}
a := |Dw|^{-1} Dw \quad \mbox{and} \quad b := \left[ \genfrac{}{}{0pt}{}{0 \ -1}{1 \hfill 0} \right] a.
\end{equation*}
Since $\{ |Du^{-1}| \geq s \}$ is not convex, we may restrict to a further subdomain to obtain,
\begin{equation*}
u_b \cdot (u_a)_b < 0 \quad \mbox{in } U.
\end{equation*}
Translating and rescaling, we may assume that $w$ has a strict maximum on the ball $\bar B_1(0) \subseteq U$ at a point $x^* \in \partial B_1(0)$.

Let $\varphi$ be the standard bump function
\begin{equation*}
\varphi(x) := \begin{cases} e^{- (1 - x^2)^{-2}} & \mbox{if } |x| < 1, \\
0 & \mbox{if }|x| \geq 1, \end{cases}
\end{equation*}
and consider the perturbation
\begin{equation*}
v := u + s \varphi u_a,
\end{equation*}
for some $s > 0$ to be determined.

Observe that
\begin{equation*}
|Dv|^2 = |Du|^2 + 2 s \max \{ u_a \cdot (\varphi u_a)_a, u_b \cdot (\varphi u_a)_b \} + O(|D ( s \varphi u_a)|^2),
\end{equation*}
and that
\begin{equation*}
u_a \cdot (\varphi u_a)_a = \varphi_a |u_a|^2 + \varphi u_a \cdot (u_a)_a,
\end{equation*}
and
\begin{equation*}
u_b \cdot (\varphi u_a)_b = \varphi u_b \cdot (u_a)_b.
\end{equation*}
Since $\phi_a(x) / \phi(x) \to - \infty$ as $x \to x^*$, we see that both terms in the max above are negative in a neighborhood of $x^*$. Thus we may choose a small $s > 0$ and obtain a $v$ that is tighter than $u$.

($\Leftarrow$) If $u$ is affine, then tightness is obvious. Thus we may suppose that $u$ is not affine and that \eqref{e.uderivatives} holds. Suppose for contradiction that $v : \bar U \to \R^2$ is tighter than $u$ and define
\begin{equation*}
s := \sup \{ \Lu u : u \neq v \} \geq \sup \{ \Lu u : \Lu u > \Lu v \},
\end{equation*}
and
\begin{equation*}
t := \sup \{ \Lu v : \Lu u < \Lu v \},
\end{equation*}
Note that $t < s$, as $v$ is tighter than $u$.

Since the singularities of \eqref{e.uderivatives} are removable, we must have $u'' \neq 0$ whenever $u' \neq 0$.  (One can easily check that if the power series expansion at a point has a term of degree $1$, no term of degree $2$, and a next term of degree $3$ or higher, then  \eqref{e.uderivatives} must have a non-removable singularity.)   Since $(\log u')' = u'/u''$, it follows that the non-zero level sets of $|u'|$ are regular.  Also, $u = v$ on $\{ \Lu u \geq s \}$.  By modifying the domain $U$, we may assume that $u \neq v$ and $\Lu u < s$ in $U$.

Suppose that $\Gamma \subseteq \partial U$ is a connected component of $\{ x \in \partial U : |u'(x)| = s \}$. Here we are using the fact that $u$ is analytic on a neighborhood of $\bar U$ to make sense of $|u'|$ on $\partial U$. Select $\ep > 0$ such that $s - \ep > t$ and let $V$ be the connected component of $\{ |u'| > s - \ep \}$ whose closure contains $\Gamma$. Since $\log |u'|$ is harmonic, $V$ is simply connected. Let $\gamma := U \cap \partial V$ be the part of the boundary of $V$ that lies in the interior of $U$. Note that $\gamma$ is a smooth curve, as the $s - \ep$ level set of $|u'|$ is regular.

Since $|u'| \geq s - \ep$ on $\bar V$, there is an $r > 0$ such that $u$ is one-to-one on $B(x,r)$ and $u(B(x,r))$ is convex for each $x \in \bar V$.  Making $\ep > 0$ smaller, we may also assume that $\bar V \subseteq \bigcup_{x \in \gamma} B(x,r)$. The differential inequality \eqref{e.delta1} implies $u(\gamma \cap B(x,r))$ is a smooth and concave part of the boundary of $u(V \cap B(x,r))$ for every $x \in \gamma$.

Let $x$ be the midpoint of $\gamma$ and consider the tangent line $l$ to $u(\gamma)$ at $u(x)$. Let $\tilde \Gamma \subseteq V$ be a maximal smooth curve such that $x \in \tilde \Gamma$ and $u(\tilde \Gamma) \subseteq l$. Observe that if $x \in \gamma$, then $l \cap u(B(x,r))$ is either empty or a single line segment. Thus $\tilde \Gamma \cap B(x,r)$ is a subset of the simple curve $(u |_{B(x,r)})^{-1}(l \cap u(B(x,r)))$. Since $\bar V \subseteq \bigcup_{x \in \gamma} B(x,r)$ and $V$ is simply connected, it follows that $\tilde \Gamma$ is simple. Now the concavity of $u(\gamma \cap B(x,r))$ guarantees that the endpoints of $\tilde \Gamma$ lie on $\partial U$.

Let $(y,z) := u(\tilde \Gamma)$ and $w := (u|_W)^{-1}$. Since $w((y,z)) = \tilde \Gamma \subseteq V$, we see that $v \circ w$ is a Lipschitz contraction on $(y,z)$. Since $w(y), w(z) \in \partial U$, we see that $v \circ w$ is the identity on $\{ y, z \}$. It follows that $w \circ v$ is the identity on $(y, z)$. In particular, $u = v$ on $\tilde \Gamma$, contradicting the fact that $u \neq v$ in $U$.
\end{proof}


\subsection{Infinity harmonic fans}
\label{secFans}
In this section, we establish that when $n=2$ smooth solutions of \EQ{pdePr} are what we call \emph{infinity harmonic fans}.  Informally, a such a fan $u$ is constructed by mapping each principal flow curve of the infinity harmonic function $f$ onto a line in $\R^m$ in such a way that motion orthogonal to $Df$ is not amplified too much.

For a concrete example, take any infinity harmonic function $f \in C^2(U, \R)$ that satisfies $\inf_U |Df| > 1$ and consider the map $u : \R^{n + m} \to \R^{1 + m}$ given by
\begin{equation*}
u : (x,y) \mapsto (f(x), y).
\end{equation*}
While $f$ maps its gradient flow lines to $\R$, the map $u$ maps the gradient flow lines of $f$ to different lines.  The following is a more general proposition that in particular implies that the example above is tight.

\begin{proposition}
\label{Fan}
Suppose $U \subseteq \R^n$ is bounded open and $u \in C^\infty(U, \R^m)$ has the form
\begin{equation}
\label{fanDec}
u(x) = f(x) b(x) + c(x),
\end{equation}
where
\begin{enumerate}
\item the function $f \in C^\infty(U, \R)$ is infinity harmonic,
\item the unit vector field $b \in C^\infty(U, \R^m)$ is constant on each flow curve of $f$,
\item the vector field $c \in C^\infty(U, \R^m)$ is constant on each flow curve of $f$,
\item the vector field $a := |Df|^{-1} Df$ is a principle direction field for $u$.
\end{enumerate}
Then $u$ is tight.
\end{proposition}
\begin{proof}
Using $b_a = 0$, $c_a = 0$, and $(Df) a_a = 0$, compute
\begin{equation*}
u_a = b f_a \mmbox{and} (u_a)_a = b (f_a)_a = b f_{aa} = 0.
\end{equation*}
Here we have again used the shorthand $v_a := v_{x_i} a^i$ and $v_{aa} := v_{x_i x_j} a^i a^j$. Thus \THM{Pr} implies that $u$ is tight.
\end{proof}

We remark that condition (4) in \PROP{Fan} is equivalent to
\begin{equation*}
|Df| = |b \otimes Df + f Db + Dc| > |f Db + Dc| \quad \mbox{in } U.
\end{equation*}
The next proposition shows that when $n=2$ and $U$ is simply connected, all non-conformal smooth tight functions have this form.

\begin{proposition}
Suppose $U \subseteq \R^2$ is bounded and simply connected and $u \in C^\infty(U, \R^m)$ is tight with principal direction field $a \in C^\infty(U, \R^n)$. Then $u$ can be written in the form \EQ{fanDec} described in Proposition \ref{Fan}.
\end{proposition}

\begin{proof}
If $u$ is tight and has a principal direction field $a \in C^\infty(U, \R^n)$, then Theorem \ref{Pr} implies that $u$ satisfies \EQ{pdePr}. In particular, the images of the principal flow lines are straight lines.  This suggests we simply define $b := |u_a|^{-1} u_a$. Now pick an arbitrary point on $x$ on one of the principal flow curves, and consider the path $P$ orthogonal to the principal flow curves that passes through $x$.  Observe that once we define $c(x)$, then \EQ{fanDec} determines $f(x)$ on the entire flow curve passing through $x$. Since $f$ must be constant on $P$, this determines $c$ and $f$ on the union of the flow lines passing through $P$. Picking a new $x$ and repeating this process allows us to construct $c$ and $f$ on all of $U$.

Observe that the $f$ and $c$ constructed as above must necessarily be smooth. Indeed, the value of $f$ as some new point $y$ is determined by the length of the path through the field $a$ to $P$. Since $a$ is smooth, this distance must be a smooth function. Since $c = u - f b$, it too must be smooth.

Finally, observe that $(u_a)_a = 0$ implies $f_{aa} = 0$. Thus $f$ is infinity harmonic.
\end{proof}

\section{Radially symmetric examples on $\R^2$}
\label{secEG}

\subsection{Classifying radially symmetric tight functions}

Let $D \subseteq \C \setminus \{ 0 \}$ be bounded and simply connected and fix $\alpha > 0$.  By \THM{Co} (and the subsequent discussion) the maps $u(z)= z^\alpha$ and $u(z) = (\bar z)^{-\alpha}$ are tight when the exponent is in $\R \setminus (0,1)$.  These maps have some nice radial symmetries: in particular, $\arg u(z) = \alpha \arg z$ for all $z\in D$ (and some branch of $\arg$ defined on $D$) and $|u(z)|$ depends only on $|z|$.  What other tight functions have these symmetries?

To frame this question more precisely, let $\rho : C_u \to A$ be the universal cover of the annulus
\begin{equation*}
A := \{ x \in \R^2 : R_1 < |x| < R_2 \},
\end{equation*}
where $0 < R_1 < R_2 < \infty$. Note that $C_u$ can be parameterized by the set $(R_1, R_2) \times \R$ so that the covering map $\rho : C_u \to A$ is given by
\begin{equation*}
\rho(r,\theta) = (r \cos \theta, r \sin \theta).
\end{equation*}
We fix $\alpha > 0$ and restrict our attention to maps $u : \bar C_u \to \R^2$ that are $\alpha$-periodic in the sense that
\begin{equation*}
u(r, \theta) = u(r, \theta + 2 \pi / \alpha) \quad \mbox{for every } r \in [R_1, R_2] \mbox{ and } \theta \in \R.
\end{equation*}
We can think of such maps $u$ as being defined on the Riemannian manifold $C$ given by $C_u$ modulo the map $(r,\theta) \to (r,\theta+2\pi/\alpha)$. In particular, this allows us to think of $z \to z^\alpha$ as an injective map from $C$ to $\R^2$.

We adopt the same definition of tight as before, even though $\bar C$ is a manifold that is not strictly a subset of $\C$ (unless $\alpha = 1$).  We now seek tight functions $u : \bar C \to \R^2$ of the form \begin{equation} \label{e.radialudef} u(r,\theta) = \rho(\phi(r), \alpha \theta).\end{equation}

Observe that \THM{Pr} and \THM{Co} have straightforward generalizations to functions $u : C \to \R^2$. \THM{Pr} implies that the map $u : C \to \R^2$ given by \EQ{e.radialudef} with
\begin{equation}
\label{prform}
\phi(r) = k r + k_0
\end{equation}
is tight provided that
\begin{equation} \label{prformreq} |k| > |\alpha|(k r + k_0) / r \mbox{ for every } r \in (R_1, R_2).
\end{equation}
Indeed, the condition on the constants $k, k_0, \alpha \in \R$ implies that the radial direction in principal. \THM{Co} implies that the map $u : C \to \R^2$ given by \EQ{e.radialudef} with
\begin{equation}
\label{coform}
\phi(r) = k r^{\pm \alpha}
\end{equation}
is tight if and only if $k = 0$ or the exponent $\pm \alpha$ is not in $(0,1)$.


\begin{proposition}
\label{Glue}
Suppose $u : C \to \R^2$ has the form \EQ{e.radialudef}.  Then $u$ is tight if and only if $\phi$ either has the form of \EQ{prform} (satisfying \EQ{prformreq}) or \EQ{coform} on all of $[R_1,R_2]$, or this is not the case and there is some $\tilde R \in (R_1,R_2)$ such that one of the following is true (here $k_i$ are assumed to be positive):
 \begin{enumerate}
\item $\alpha \geq 1$ and $\phi$ has the form $k_1 r^{- \alpha}$ on $[R_1,\tilde R]$ and the form $k_2 r^{\alpha}$ on $[\tilde R, R_2]$
\item $\alpha \geq 1$ and $\phi$ has the form $k_1 r^{- \alpha}$ on $[R_1,\tilde R]$ and is a decreasing affine function on $[\tilde R, R_2]$ (with matching slope at $\tilde R$).
\item $\alpha \geq 1$ and $\phi$ has the form $k_2 r^{\alpha}$ on $[\tilde R, R_2]$ and is an increasing affine function on $[R_1,\tilde R]$ (with matching slope at $\tilde R$).
\item $\alpha \in (0,1)$ and $\phi$ has the form $k_1 r^{-\alpha}$ on $[R_1, \tilde R]$ and is an increasing affine function on $[\tilde R,R_2]$ (with opposite slope at $\tilde R$).
\item $\alpha \in (0,1)$ and $\phi$ has the form $k_1 r^{-\alpha}$ on $[R_1, \tilde R]$ and is a decreasing affine function on $[\tilde R,R_2]$ (with matching slope at $\tilde R$).
\end{enumerate}
\end{proposition}

The five possibilities are illustrated in Figure \ref{joints}.  The lighter curves in the background illustrate the functions $k r^{\pm \alpha}$ for a range of $k$ values.  Note that if $\phi$ is one of the functions described in Proposition \ref{Glue}, then on each the two intervals $[R_1, \tilde R]$ and $[\tilde R, R_2]$ it either traces one of the curves $k r^{\pm \alpha}$ or it traces a straight line that is strictly steeper than each such curve it intersects (at the point of intersection).  Note that $\alpha \geq 1$ for the first three figures and $\alpha \in (0, 1)$ for the fourth and fifth (hence the concavity of the increasing background curves).

\begin{figure}[h]
\caption{The five possible interfaces in our radially symmetric examples.}
\label{joints}
\bigskip
\includegraphics[width=4in]{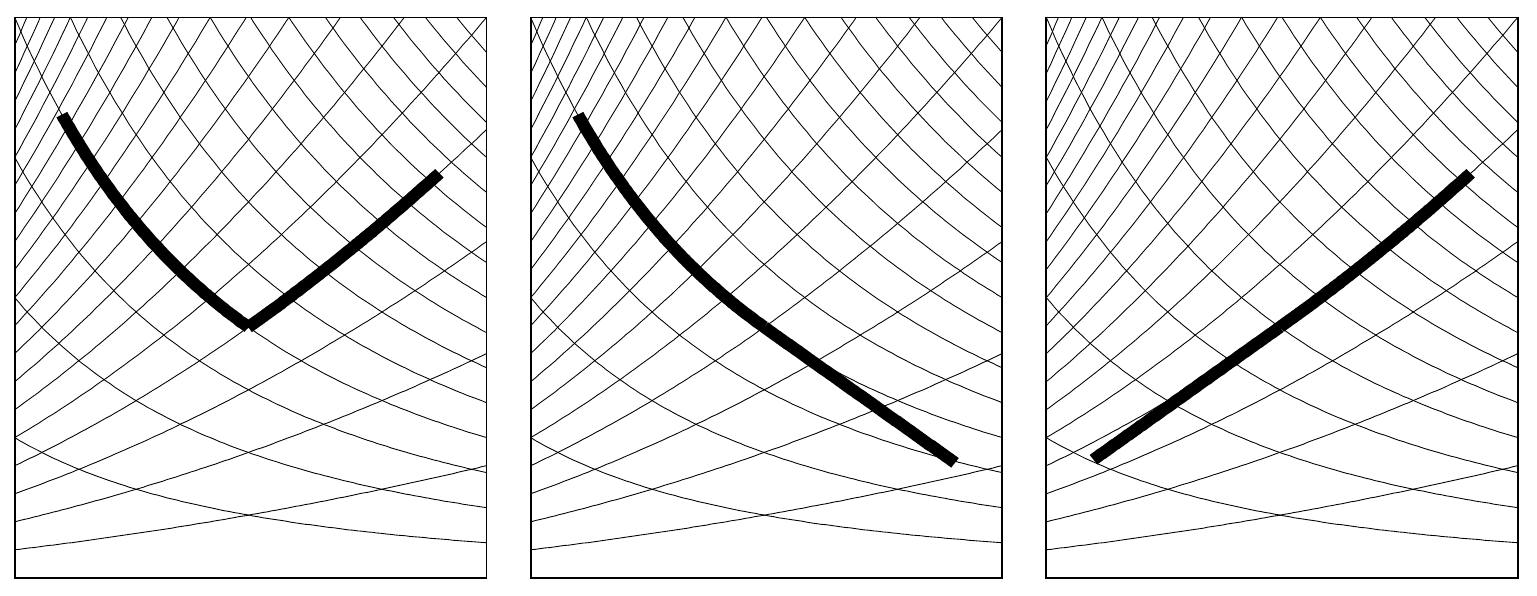}
\includegraphics[width=2.7in]{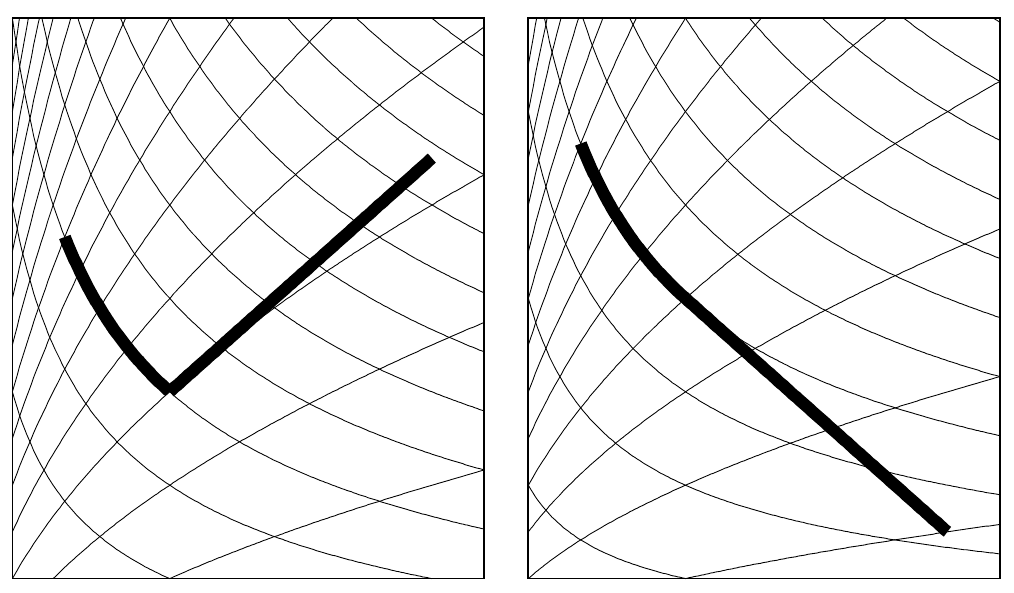}
\end{figure}

\begin{corollary}
If $g : \partial C \to \R^2$ has the form
\begin{equation}
\label{gform}
g(r, \theta) = \left\{ \begin{array}{ll}
\rho(M_1, \alpha \theta) & \mbox{if } r = R_1, \\
\rho(M_2, \alpha \theta) & \mbox{if } r = R_2, \\
\end{array} \right.
\end{equation}
for some $\alpha \in \R$, then $g$ has a unique tight extension $u : C \to \R^2$ of the form \EQ{e.radialudef}.
\end{corollary}
\begin{proof}
A straightforward inspection shows that any pair of points $(R_1,M_1)$ and $(R_2,M_2)$ in $(0,\infty)^2$, with $R_1 \not = R_2$, can be joined by a unique function $\phi$ that either has one of the four types shown or is of the type \EQ{prform} (satisfying \EQ{prformreq}) or \EQ{coform} on all of $[R_1,R_2]$.  This can be seen by viewing the background curves $k r^{\pm \alpha}$ as the lines of a new coordinate system: that is, we may change coordinates via $(R,M) \to (R^\alpha/M, R^{-\alpha}/M)$.  In the case $\alpha \geq 1$, the difference between $(R_2,M_2)$ and $(R_1,M_1)$ (in these new coordinates) can lie in three possible quadrants, and the three figures represent these three cases.  In the case $\alpha < 1$, the two figures shown represent two possible quadrants; the third quadrant figure is not shown, since one always has a straight line in that case.  The corollary then follows from Proposition \ref{Glue}.
\end{proof}

\begin{lemma}
\label{radAvg}
A Lipschitz function $u \in C(C , \R^2)$ of the form \EQ{e.radialudef} is tight if and only if it there is no tighter $v$ of the form \EQ{e.radialudef}.
\end{lemma}

\begin{proof}
Suppose $v \in C(C, \R^2)$ is tighter than $u$, i.e.,
$$M_1:=\sup \{ \Lu u : \Lu v < \Lu u \} > \sup \{ \Lu v : \Lu v > \Lu u \}=: M_2.$$
We define a rotated version of $v$ by $v_\beta(r,\theta) = e^{-i\alpha\beta} v(r, \theta + \beta)$ where ($\theta+ \beta$ is computed modulo $2\pi/\alpha$).  Note that if $v$ has the form  \EQ{e.radialudef} then $v_\beta = v$.  Otherwise, we may consider the symmetrized function $w := \int_0^{2\pi/\alpha} v_\beta d\beta$.  We claim that if $v$ is tighter than $u$ then $w$ is also tighter than $u$.  To see this note first that on any circle $(r, \cdot)$ we have:
\begin{enumerate}
\item $\Lu u$ constant,
\item $\Lu v$ possibly non-constant, but never larger than $\Lu u$ if $\max \{\Lu u, \Lu v \} \geq M_2$,
\item $\Lu w \leq \Lu u$ if $\max \{\Lu u, \Lu w \} \geq M_2$ (by Jensen's inequality).
\end{enumerate}
This implies \begin{equation} \label{vsup1} \sup \{ \Lu w : \Lu w > \Lu u \} \leq M_2.\end{equation}  Similarly, on some circle $(r, \cdot)$ we must have
\begin{enumerate}
\item $\Lu u$ constant with $\Lu u > M_2$,
\item $\Lu v$ possibly non-constant with $\Lu v \leq \Lu u$ on the whole circle and $\Lu v < \Lu u$ on a positive measure subset,
\item $\Lu w < \Lu u$ (by Jensen's inequality).
\end{enumerate}
This implies \begin{equation} \label{vsup2} \sup \{ \Lu u : \Lu u > \Lu w \} > M_2.\end{equation}  Now \EQ{vsup1} and \EQ{vsup2} together imply that $w$ is tighter than $u$.  Now, it is not necessarily the case that $w$ has the form \EQ{e.radialudef}.  The radial symmetries show only that it can be written as
$$w(r,\theta) = \rho(\phi(r), \alpha \theta + \alpha_0(r)),$$
where $\alpha_0$ is some function of $r$.  Let $\tilde w$ be the function obtained from $w$ by replacing this $\alpha_0$ with $0$.  Clearly, $\Lu \tilde w \leq \Lu w$ pointwise, and hence $\tilde w$ is also tighter than $u$.
\end{proof}


\begin{proof}[{\bf Proof of \PROP{Glue}}]
By \LEM{radAvg}, it is enough to consider functions of the form \EQ{e.radialudef}, and to show that \begin{enumerate}
\item If $\phi$ is not of one of the types in the proposition statement, then one can modify $\phi$ (keeping same boundary data) in a way that makes \EQ{e.radialudef} tighter.
\item If $\phi$ is of one of the types in the proposition statement, then one cannot do this.
\end{enumerate}

To establish (1), suppose we are given boundary conditions $(R_1,M_1)$ and $(R_2,M_2)$, let $\bar \phi$ be the interpolation of Proposition \ref{Glue}.  Simple inspection shows that if $\phi$ is strictly larger than $\bar \phi$ on $(R_1, R_2)$ then the extension \EQ{e.radialudef} becomes tighter if $\phi$ is replaced by $\bar \phi$ (note that the maximal contribution to $\Lu u$ comes at one of the endpoints of the interval).  This argument in fact shows that in order for $u$ to be tight, if $\phi$ hits the boundary points $(R_1,M_1)$ and $(R_2,M_2)$ then it must be equal to or less than $\bar \phi$ on $(R_1, R_2)$ (if it is larger on some open interval, we may replace $R_1$ and $R_2$ with the endpoints of that interval and apply the above).

Next, if it happens that $\bar \phi$ is of the type \EQ{prform} (satisfying \EQ{prformreq}) then if $\phi$ is any function other than $\bar \phi$ then $u$ becomes tighter when $\phi$ is replaced by $\bar \phi$ (simply because the Lipschitz norm of a one dimensional function on a finite interval --- with given boundary data -- is minimized by a straight line).  This argument shows that if the graph of $\phi$ hits two points on such a $\bar \phi$, and $\phi$ is not affine between those points, then $u$ is not tight.

Now we claim that if the endpoints of the graph of $\phi$ lie on a convex $kr^{\pm \alpha}$, and the $\phi$ goes below it, then $\phi$ is not tight.  Considering first the $r^\alpha$ case, the function $\phi(r)/r^{\alpha}$ (which is Lipschitz, hence a.e. differentiable) has a positive derivative at some point $r'$.  This means that $\phi'(r)$ is steeper than than $kr^{\alpha}$ curve through $(r, \phi(r))$, and the previous argument implies that if $\phi$ is tight then $\phi$ must be equal to an affine function $g$ in a neighborhood of such a point.  Taking $r''$ to be the smallest point at which $\phi(r)=g(r)$, we find that for some point $r'''$ just smaller than $r''$ the interval $(r''',r')$ is one on which $\phi$ is not affine, even though the corresponding $\bar \phi$ is, and we have reduced to the previous case.  The $r^{-\alpha}$ case is similar.

Now we know that if $\phi$ is tight and strictly below $\bar \phi$ then $\phi/\bar\phi$ cannot obtain a minimum anywhere except $\tilde R$, and that in this case $\phi$ must be affine on both $[R_1, \tilde R]$ and $[\tilde R, R_2]$.  It is easy to see by inspection that in this case $\bar \phi$ is tighter than $\phi$.  This completes the proof of (1).

To establish (2), we must show that the $\bar \phi$ are in fact tight.  If $\phi$ yielded a tighter function, then there would be some interval such that $\phi$ and $\bar \phi$ were equal on the endpoints, not equal inside, and had the restriction of $\phi$ to the interval tighter than $\bar \phi$ on the interval.  This cannot happen if $\phi > \bar \phi$, since in this case the corresponding $\bar u$ would have a strictly higher Lipschitz constant in a neighborhood of the boundary, where the maximum is obtained.  If $\phi < \bar \phi$ then the
derivative of $\phi$ (which again exists a.e.) would have to be strictly less than that of $\phi$ at points arbitrarily near $R_1$ (and greater at points arbitrarily near $R_2$); thus, at points arbitrarily close to the endpoint where $\bar \phi$ is steepest, we have $\phi$ even steeper.  We conclude that $\phi$ cannot be tighter than $\bar \phi$.
\end{proof}

\section{Questions}
\label{secQ}

\begin{question}
Does every Lipschitz function $g : \partial U \to \R^m$ on the boundary of a bounded open set $U \subseteq \R^n$ admit a (unique) tight extension?

One approach to proving existence and uniqueness in the scalar case is to establish good estimates for discrete infinity harmonic functions on lattice models (see, e.g., Lemma 3.9 in \cite{armstrong-2009}). It may be possible to use tight functions on graphs to do this in the vector-valued case. However, one must be careful to avoid the instability of uniqueness described above.
\end{question}

\begin{question}
By \PROP{Fan}, we know that fans of smooth infinity harmonic functions are tight. What happens if we fan out a non-smooth infinity harmonic functions like the Aronsson function $(x,y) \mapsto x^{4/3} - y^{4/3}$?
\end{question}

\begin{question}
Suppose $U \subseteq \R^2$ and $u \in C^2(U, \R^2)$ is tight. Suppose, moreover, that the interface between the regions where $u$ is conformal and $u$ has a principal direction is a smooth curve. What can be said about $u$ along that interface?
\end{question}

\begin{question}
What kinds of surfaces can be images of tight functions in the $\R^2 \to \R^m$ case?
\end{question}

\begin{question}
For any differentiable map $u$, we may denote by $S_k$ the set of $x$ for which the eigenspace of the largest eigenvalue of $Du(x)^tDu(x)$ is $k$ dimensional. Informally, $S_k$ is the set of locations where there are $k$ principle directions.  What can one say in general about the behavior of $u$ on $S_k$?  Can one construct any non-trivial examples of tight functions for which each of the $S_k$ has non-empty interior?
\end{question}

\bibliographystyle{amsplain}
\bibliography{2damle}

\end{document}